\documentclass[11pt,reqno]{amsart}
\usepackage{amssymb,amsmath,amsthm,fullpage}
\usepackage{verbatim}
\usepackage{graphicx}
\newtheorem{theorem}{Theorem}
\newtheorem*{thm:main}{Theorem \ref{thm:main}}
\newtheorem{lemma}[theorem]{Lemma}

\newtheorem{conjecture}[theorem]{Conjecture}
\newtheorem{claim}{Claim}
\newtheorem{proposition}[theorem]{Proposition}

\newcommand{\mc}[1]{\mathcal{#1}}

\newcommand{\aA}{\alpha}
\newcommand{\dD}{\delta}

\newcommand{\tT}{\theta}
\newcommand{\pl}{\partial}
\begin{document}
\title{An approximate isoperimetric inequality for \(r\)-sets}
\author{Demetres Christofides, David Ellis and Peter Keevash}
\thanks{The research of the first author was supported by EPSRC grant EP/G056730/1; the research of the third author was supported in part by ERC grant 239696 and EPSRC grant EP/G056730/1.}
\date{March 2012}
\maketitle
\begin{abstract}
We prove a vertex-isoperimetric inequality for \([n]^{(r)}\), the set of all \(r\)-element subsets of \(\{1,2,\ldots,n\}\), where \(x,y \in [n]^{(r)}\) are adjacent if \(|x \Delta y|=2\). Namely, if \(\mathcal{A} \subset [n]^{(r)}\) with \(|\mathcal{A}|=\alpha {n \choose r}\), then its vertex-boundary \(b(\mathcal{A})\) satisfies
\[|b(\mathcal{A})| \geq c\sqrt{\frac{n}{r(n-r)}} \alpha(1-\alpha) {n \choose r},\]
where \(c\) is a positive absolute constant. For \(\alpha\) bounded away from 0 and 1, this is sharp up to a constant factor (independent of \(n\) and \(r\)).
\end{abstract}

\section{Introduction}
Isoperimetric problems are classical objects of study in mathematics. In general, they ask for the smallest possible `boundary' of a set of a certain `size'. For example, of all shapes in the plane with area 1, which has the smallest perimeter? The ancient Greeks `knew' that the answer was a circle, but it was not until the 19th century that this was proved rigorously.

In the last fifty years, discrete isoperimetric problems have been extensively studied. These deal with notions of boundary in graphs. Here, there are two competing notions of boundary. If \(G=(V,E)\) is a graph, and \(S \subset V\), the {\em vertex-boundary of \(S\) in \(G\)} is the set of all vertices in \(V \setminus S\) which have a neighbour in \(S\). Similarly, the {\em edge-boundary of \(S\) in \(G\)} is the set of all edges of \(G\) between \(S\) and \(V \setminus S\). The {\em vertex-isoperimetric problem for \(G\)} asks for the minimum possible size of the vertex-boundary of a \(k\)-element subset of \(V\), for each \(k \in \mathbb{N}\). Similarly, the {\em edge-isoperimetric problem for \(G\)} asks for the minimum possible size of the edge-boundary of a \(k\)-element subset of \(V\), for each \(k \in \mathbb{N}\).

The reader is referred to \cite{bezrukov,compressions} for surveys of discrete isoperimetric inequalities from a combinatorial perspective, and to \cite{talagrand} for a discussion of the connection with concentration of measure, and for several applications, notably in geometric probability and percolation theory.

A fundamental example arises from taking our graph \(G\) to be the \(n\)-dimensional hypercube \(Q_n\), the graph on \(\{0,1\}^n\) where \(x\) and \(y\) are adjacent if they differ in exactly one coordinate. It turns out that the edge-boundary of a \(k\)-element set is minimized by taking the first \(k\) elements of the binary ordering on \(\{0,1\}^n\); this was proved by Harper \cite{harper}, Lindsey \cite{lindsey}, Bernstein \cite{bernstein} and Hart \cite{hart}. Hart's proof uses induction on \(n\), combined with an inequality concerning the number of 1's in initial segments of the binary ordering on \(\{0,1\}^n\).

The vertex-isoperimetric problem for \(Q_n\) was solved by Harper \cite{harper2}. To state it, we identify \(\{0,1\}^n\) with \(\mathcal{P}([n])\), the power-set of \([n]\), by identifying \(v \in \{0,1\}^n\) with the set \(\{i \in [n]:\ v_i=1\}\). Harper's Theorem states that the vertex-boundary of a \(k\)-element subset of \(\mathcal{P}([n])\) is minimized by taking the first \(k\) elements of the {\em simplicial} ordering on \(\mathcal{P}([n])\). (If \(x,y \subset [n]\), we say that \(x < y\) in the simplicial ordering if \(|x| < |y|\), or \(|x|=|y|\) and \(\min(x \Delta y) \in x\).) Note that if \(k\) is of the form \(\sum_{i=0}^{d} {n \choose i}\), then the first \(k\) elements of the simplicial ordering are simply all the subsets of \([n]\) with size at most \(d\), i.e. the Hamming ball with centre \(\emptyset\) and radius \(d\). 

Both theorems can be proved using compressions; see for example \cite{compressions}. However, this technique relies upon the fact that the extremal examples are `nested'; isoperimetric problems without this property require other techniques, and tend to be harder.

In this paper, we consider the vertex-isoperimetric problem for \(r\)-sets. As usual, let \([n]\) denote the set \(\{1,2,\ldots,n\}\), and let \([n]^{(r)}\) denote the set of all \(r\)-element subsets (or `\(r\)-sets') of \(\{1,2,\ldots,n\}\). Let \(\Gamma_{n}^{(r)}\) be the graph on \([n]^{(r)}\) where \(xy \in E(\Gamma_{n}^{(r)})\) if \(|x \Delta y| =2\). If \(\mathcal{A} \subset [n]^{(r)}\), we let \(b(\mathcal{A})\) denote the vertex-boundary of \(\mathcal{A}\) in this graph, i.e. 
\[b(\mathcal{A}) = \{x \in [n]^{(r)} \setminus \mathcal{A}:\ |x \Delta y| = 2\ \textrm{for some }y \in \mathcal{A}\}.\]
If \(r \leq k \leq n-r\) and \(1 \leq \ell \leq r\), we let
\[\mathcal{B}_{k,\ell} = \{x \in [n]^{(r)}:\ |x \cap [k]| \geq \ell\}.\]
Bollob\'as and Leader \cite{leader} report the following `folklore' conjecture.
\begin{conjecture}
\label{conj:main}
If \(\mathcal{A} \subset [n]^{(r)}\), then there exists a set \(\mathcal{C}\) with \(\mathcal{B}_{k,\ell} \subset \mathcal{C} \subset \mathcal{B}_{k,\ell-1}\) for some \(k\) and \(\ell\) with \(r \leq k \leq n-r\) and \(1 \leq \ell \leq r\), such that \(|b(\mathcal{A})| \geq |b(\mathcal{C})|\).
\end{conjecture}
We note that except in the case \(r=n/2\), the conjectured extremal sets are not nested; as mentioned above, this rules out a proof based upon compressions alone.

In this paper, we prove an approximate version of Conjecture \ref{conj:main}. Our aim is to exhibit a simple function \(F(n,r,\alpha)\) such that for any subset \(\mathcal{A} \subset [n]^{(r)}\) with \(|\mathcal{A}| = \alpha {n \choose r}\), we have \(|b(\mathcal{A})| \geq F(n,r,\alpha)\). We prove the following.

\begin{theorem}
\label{thm:main}
Let \(1 \leq r \leq n-1\). If \(\mathcal{A} \subset [n]^{(r)}\) with \(|\mathcal{A}|=\alpha {n \choose r}\), then
\[|b(\mathcal{A})| \geq c \sqrt{\frac{n}{r(n-r)}} \alpha(1-\alpha) {n \choose r},\]
where \(c\) is a postive absolute constant. (We may take \(c = 1/5\).)
\end{theorem}

Our proof uses induction on \(n\). For \(\mathcal{A} \subset [n]^{(r)}\), we will consider the {\em lower} and {\em upper \(n\)-sections} of \(\mathcal{A}\), defined respectively as
\[\mathcal{A}_{0} := \{x \in \mathcal{A}:\ n \notin x\},\quad \mathcal{A}_{1} := \{x \in [n-1]^{(r-1)}:\ x \cup \{n\} \in \mathcal{A}\}.\]
The main idea of the proof is simple: we split into cases depending on the relative sizes of the \(n\)-sections, and then we use the induction hypothesis if they are very similar in size, and the so-called Local LYM inequality otherwise.

We pause to observe that Theorem \ref{thm:main} is sharp up to a constant factor, for sets of size \(\tfrac{1}{2} {n \choose r}\). To see this, assume for simplicity that \(n\) is even and \(r\) is odd. Consider
\[\mathcal{A} = \{x \in [n]^{(r)}:\ |x \cap \{1,2,\ldots,n/2\}| > r/2\}.\]
We have \(|\mathcal{A}| = \tfrac{1}{2}{n \choose r}\) (as may be seen by considering the bijection between \(\mathcal{A}\) and \([n]^{(r)}\setminus \mathcal{A}\) given by interchanging the elements \(i\) and \(i+n/2\) for each \(i \in [n/2]\)), and
\[b(\mathcal{A}) = \{x \in [n]^{(r)}:\ |x \cap \{1,2,\ldots,n/2\}| = (r-1)/2\},\]
so
\[\frac{|b(\mathcal{A})|}{{n \choose r}} = \frac{{n/2 \choose (r-1)/2}{n/2 \choose (r+1)/2}}{{n \choose r}}=\frac{{r \choose (r-1)/2} {n -r \choose (n-r+1)/2}}{{n \choose n/2}}.\]
(To obtain the second equality above, observe that
\[{n \choose n/2}{n/2 \choose (r-1)/2}{n/2 \choose (r+1)/2} = {n \choose r}{r \choose (r-1)/2} {n -r \choose (n-r+1)/2},\]
as both sides count the number of pairs \((y,z) \in [n]^{(n/2)} \times [n]^{(r)}\) such that \(|y \cap z| = (r-1)/2\).)

Using
\[{m \choose \lfloor m/2 \rfloor} = \Theta(2^m/\sqrt{m}),\]
we obtain
\[\frac{|b(\mathcal{A})|}{{n \choose r}} =  \frac{\Theta(2^r/\sqrt{r})\Theta (2^{n-r}/\sqrt{n-r})}{\Theta(2^{n}/\sqrt{n})}=\Theta(\sqrt{\tfrac{n}{r(n-r)}}).\]

Similarly, it is easy to see that for \(\alpha \in [\alpha_0,1-\alpha_0]\), where \(\alpha_0 >0\), Theorem \ref{thm:main} is sharp up to a constant factor depending upon \(\alpha_0\) alone (i.e., indepedent of \(n\) and \(r\)). This follows from considering sets \(\mathcal{C}\) such that
\[\mathcal{B}_{\lfloor n/2 \rfloor, \lceil r/2+\beta \sqrt{r} \rceil} \subset \mathcal{C} \subset \mathcal{B}_{\lfloor n/2 \rfloor, \lfloor r/2+\beta \sqrt{r} \rfloor},\]
where \(-\gamma_0 \leq \beta \leq \gamma_0\), for some constant \(\gamma_0\) depending on \(\alpha_0\).

In order to motivate the proof of our result, in the next section we show how a similar (but easier) method can be used to obtain an approximate vertex-isoperimetric inequality for \(Q_n\).

\section{An approximate vertex-isoperimetric inequality for \(Q_n\)}
In this section, we consider subsets of \(\{0,1\}^n\), which we identify with \(\mathcal{P}([n])\) as descibed above. Note that \(x,y \subset [n]\) are joined in \(Q_n\) if and only if \(|x \Delta y|= 1\). If \(\mathcal{A} \subset \mathcal{P}([n])\), in this section (alone) we write \(b(\mathcal{A})\) for the boundary of \(\mathcal{A}\) in \(Q_n\), and \(N(\mathcal{A}) = \mathcal{A} \cup b(\mathcal{A})\) for the neighbourhood of \(\mathcal{A}\) in \(Q_n\). We will prove the following approximate vertex-isoperimetric inequality for \(Q_n\).
\begin{theorem}
\label{thm:approxcube}
Let \(n \geq 1\). If \(\mathcal{A} \subset \mathcal{P}([n])\) with \(|\mathcal{A}| = \alpha 2^n\), then
\[|b(\mathcal{A})| \geq \sqrt{2} \alpha(1-\alpha) 2^n / \sqrt{n}.\]
\end{theorem}
\begin{proof}
By induction on \(n\). It is easily checked that the theorem holds for \(n=1\). Let \(n \geq 2\), and suppose that the statement of the theorem holds for \(n-1\). Let \(\mathcal{A} \subset \mathcal{P}([n])\) with \(|\mathcal{A}| = \alpha 2^n\). Let
\[\mathcal{A}_{0} = \{A \in \mathcal{A}:\ n \notin A\},\ \mathcal{A}_{1} = \{A \in \mathcal{P}([n-1]):\ A \cup \{n\} \in \mathcal{A}\}\]
denote the lower and upper \(n\)-sections of \(\mathcal{A}\), respectively. We consider these as subsets of \(\mathcal{P}([n-1])\).

By considering \(\{x \Delta \{n\}: x \in \mathcal{A}\}\) if necessary, we may assume that \(|\mathcal{A}_0| \geq |\mathcal{A}_1|\). Define
\[\delta = |\mathcal{A}_0|/2^{n-1} - \alpha;\]
then we have \(|\mathcal{A}_0| = (\alpha + \delta)2^{n-1}\) and \(|\mathcal{A}_1| = (\alpha-\delta)2^{n-1}\), where \(\delta \geq 0\).

Observe that
\begin{align}
\label{eq:comb} |N(\mathcal{A})| &= |(N(\mathcal{A}))_{0}|+|(N(\mathcal{A}))_1| \nonumber\\
& = |N(\mathcal{A}_0) \cup \mathcal{A}_1| + |N(\mathcal{A}_1) \cup \mathcal{A}_0| \nonumber\\
& \geq \max\{2|\mathcal{A}_0|, |N(\mathcal{A}_0)|+|N(\mathcal{A}_1)|\},
\end{align}
noting that on the right-hand side, \(N\) is defined with respect to \(\mathcal{P}([n-1])\). Let \(f(x) = \sqrt{2}x(1-x)\). By the induction hypothesis, we have
\[|N(\mathcal{A}_0)| \geq |\mathcal{A}_0|+f(\alpha+\delta)2^{n-1}/\sqrt{n-1},\quad |N(\mathcal{A}_1)| \geq |\mathcal{A}_1|+ f(\alpha-\delta)2^{n-1}/\sqrt{n-1};\]
combining with (\ref{eq:comb}), we obtain:
\[|N(\mathcal{A})| \geq \max\{(\alpha+\delta)2^n,|\mathcal{A}_0|+f(\alpha+\delta)2^{n-1}/\sqrt{n-1}+|\mathcal{A}_1|+f(\alpha-\delta)2^{n-1}/\sqrt{n-1}\}.\]
In terms of the boundary, this becomes
\[|b(\mathcal{A})| \geq \max\{\delta 2^n, (f(\alpha+\delta)+f(\alpha-\delta))2^{n-1}/\sqrt{n-1}\}.\]
Hence, to complete the proof of the inductive step, it suffices to show that
\[\max\{\delta, (f(\alpha+\delta)+f(\alpha-\delta))/(2\sqrt{n-1})\} \geq f(\alpha)/\sqrt{n}.\]
This is clearly equivalent to the following.
\begin{claim}
\label{claim:warmup}
Let \(n \in \mathbb{N}\), and let \(\alpha \in (0,1)\). If \(0 \leq \delta \leq f(\alpha)/\sqrt{n}\), then
\[\frac{f(\alpha+\delta)+f(\alpha-\delta)}{2f(\alpha)} \geq \sqrt{1-1/n}.\]
\end{claim}
\begin{proof}[Proof of Claim \ref{claim:warmup}]
We have
\begin{align*}
\frac{f(\alpha+\delta)+f(\alpha-\delta)}{2f(\alpha)} &= \frac{(\alpha+\delta)(1-\alpha-\delta) + (\alpha-\delta)(1-\alpha+\delta)}{2\alpha(1-\alpha)} \\
& = 1-\frac{\delta^2}{\alpha(1-\alpha)}\\
& \geq 1-\frac{2\alpha(1-\alpha)}{n}\\
& \geq 1-1/(2n)\\
& \geq \sqrt{1-1/n},
\end{align*}
where we have used the inequality \(\sqrt{1-x} \leq 1-x/2\), which is valid for all \(x \in [0,1]\). 
\end{proof}
This completes the proof of Theorem \ref{thm:approxcube}.
\end{proof}
Observe that theorem \ref{thm:approxcube} is sharp up to a constant factor, for half-sized sets. To see this, assume for simplicity that \(n\) is odd, and consider the set
\[\mathcal{A} = \{x \in \mathcal{P}([n]):\ |x| \leq n/2\}.\]
This has \(|\mathcal{A}| = 2^{n-1}\), and \(b(\mathcal{A}) = [n]^{(n+1)/2}\), so
\[|b(\mathcal{A})| = {n \choose (n+1)/2} = (1+o(1))\sqrt{\tfrac{2}{\pi n}} 2^n.\]
As \(\alpha \to 0\), however, the ratio between the bound in Theorem \ref{thm:approxcube} and the exact bound given by Harper's Theorem tends to zero.

Note that Theorem \ref{thm:approxcube} follows from a result of Talagrand in \cite{talagrand2}. For \(\mathcal{A} \subset \mathcal{P}([n])\), Talagrand defines
\[h_{\mathcal{A}}(x) = \left\{\begin{array}{ll}|\{i \in [n]:\ x \Delta \{i\} \notin \mathcal{A}\}| & \textrm{if }x \in \mathcal{A},\\
0 & \textrm{if }x \notin \mathcal{A}.\end{array}\right.\]
The quantity \(h_{\mathcal{A}}\) can be seen as a measure of the `surface area' of \(\mathcal{A}\). Using an induction argument somewhat similar to the proof above, Talagrand proves the following.
\begin{theorem}[Talagrand, \cite{talagrand2}]
\label{thm:talagrand}
If \(\mathcal{A} \subset \mathcal{P}([n])\) with \(|\mathcal{A}| = \alpha 2^n\), then
\[\frac{1}{2^n} \sum_{x \subset [n]} \sqrt{h_{\mathcal{A}}(x)} \geq \sqrt{2} \alpha (1-\alpha).\]
\end{theorem}
Noting that \(h_{\mathcal{A}}(x) \leq n\) for all \(x\), and applying Theorem \ref{thm:talagrand} to \(\mathcal{P}([n])\setminus \mathcal{A}\), yields Theorem \ref{thm:approxcube}.

Bobkov \cite{bobkov} proves an isoperimetric inequality involving a different notion of surface area, again using a somewhat similar induction argument.
\section{Proof of Theorem \ref{thm:main}}
In this section, we will prove our main theorem.

\begin{thm:main}
Let \(n \geq 2\) and let \(1 \leq r \leq n-1\). If \(\mathcal{A} \subset [n]^{(r)}\) with \(|\mathcal{A}|=\alpha {n \choose r}\), then
\[|b(\mathcal{A})| \geq c\sqrt{\frac{n}{r(n-r)}}\alpha(1-\alpha) {n \choose r},\]
where \(c\) is a positive absolute constant. (We may take \(c = 1/5\).)
\end{thm:main}
\begin{proof}
Firstly, observe that if the theorem holds for the pair \((n,r)\) (for all subsets \(\mathcal{A}\)), then it holds for the pair \((n,n-r)\) also. (Replace \(\mathcal{A}\) by \(\{[n] \setminus x:\ x \in \mathcal{A}\}\).) Secondly, observe that the theorem holds for \(r=1\) and all \(n\), provided we take \(c \leq 1\). Indeed, suppose that \(r=1\). Since all singletons are adjacent to one another, the boundary of a set \(\mathcal{A} \subset [n]^{(1)}\) is precisely \([n]^{(1)}\setminus \mathcal{A}\). Hence, if \(|\mathcal{A}| = \alpha n\), then \(|b(\mathcal{A})| = (1-\alpha)n\). Moreover, if \(\mathcal{A} \neq [n]^{(1)}\), then \(\alpha \leq 1-1/n\). Hence,
\[\frac{|b(\mathcal{A})|}{c\sqrt{\tfrac{n}{n-1}}\alpha(1-\alpha)n} = \frac{(1-\alpha)n}{c\sqrt{\tfrac{n}{n-1}}\alpha(1-\alpha)n} = \tfrac{1}{c\alpha} \sqrt{1-\tfrac{1}{n}} \geq \frac{1}{1-\tfrac{1}{n}} \sqrt{1-\tfrac{1}{n}} = \sqrt{\tfrac{n}{n-1}} > 1,\]
proving the theorem in this case.

It follows from these two observations that the theorem holds for all \(n \leq 3\), provided we take \(c \leq 1\).

We now proceed by induction on \(n\). Let $n \geq 4$, and let \(\mathcal{A} \subset [n]^{(r)}\) with \(|\mathcal{A}| = \alpha {n \choose r}\). By the above observations, we may assume that \(2 \leq r \leq n/2\). Assume that the statement of the theorem holds for \(n-1\).

We introduce the following notation. For any family \(\mathcal{B} \subset [n]^{(r)}\), let
\[\mathcal{B}_{0} = \{x \in \mathcal{B}:\ n \notin x\} \subset [n-1]^{(r)}
\quad \text{ and } \quad
\mathcal{B}_{1} = \{x \in [n-1]^{(r-1)}:\ x \cup \{n\} \in \mathcal{B}\}\]
denote the lower and upper \(n\)-sections of \(\mathcal{B}\), respectively. 
Let
\[ \pl^-\mc{B} = \{ x \in [n]^{(r-1)}: \exists y \in \mc{B} \textrm{ with } x \subset y \} 
\quad \text{ and } \quad
\pl^+\mc{B} = \{ x \in [n]^{(r+1)}: \exists y \in \mc{B} \textrm{ with } y \subset x \} \]
denote the lower and upper shadows of \(\mc{B}\), respectively.
We let
\[\alpha_0 = \frac{|\mathcal{A}_0|}{{n-1 \choose r}}
\quad \text{ and } \quad
\alpha_1 = \frac{|\mathcal{A}_1|}{{n-1 \choose r-1}}.\]
Since \(|\mathcal{A}| = |\mathcal{A}_0|+|\mathcal{A}_1|\), we have
\[\alpha{n \choose r} = \alpha_0 {n-1 \choose r}+\alpha_1 {n-1 \choose r-1},\]
and so
\[\alpha = (1-\tfrac{r}{n}) \alpha_0 + \tfrac{r}{n}\alpha_1.\]

Let \(N(\mathcal{A}) = \mathcal{A} \cup b(\mathcal{A})\) denote the neighbourhood of \(\mathcal{A}\).  
Observe that
\[|N(\mathcal{A})| = |(N(\mathcal{A}))_{0}|+|(N(\mathcal{A}))_1| 
= |N(\mathcal{A}_0) \cup \partial^{+}(\mathcal{A}_1)| + |N(\mathcal{A}_1) \cup \partial^{-}(\mathcal{A}_0)|,\]
noting that on the right-hand side, $N$, $\pl^-$, $\pl^+$ are defined with respect to $[n-1]$.

It follows that
\begin{equation}\label{eq:roughbound}
|N(\mathcal{A})| \geq \max\{
|N(\mathcal{A}_0)|+|\partial^{-}(\mathcal{A}_0)|,
|N(\mathcal{A}_1)|+|\partial^{+}(\mathcal{A}_1)|,
|N(\mathcal{A}_0)|+|N(\mathcal{A}_1)| \}.
\end{equation}
Next we recall the `Local LYM inequality' (see for example \cite[Chapter 3]{bollobas}): if \(\mathcal{B} \subset [m]^{(k)}\), then
\[\frac{|\partial^{-} \mathcal{B}|}{{m \choose k-1}} \geq \frac{|\mathcal{B}|}{{m \choose k}}
\quad \text{ and } \quad 
\frac{|\partial^{+} \mathcal{B}|}{{m \choose k+1}} \geq \frac{|\mathcal{B}|}{{m \choose k}}.\]
It follows that
\[|\partial^{-}(\mathcal{A}_0)| \geq \tfrac{r}{n-r}|\mathcal{A}_0|
\quad \text{ and } \quad 
|\partial^{+}(\mathcal{A}_1)| \geq \tfrac{n-r}{r} |\mathcal{A}_1|.\]
Substituting this into (\ref{eq:roughbound}) gives
\begin{equation}\label{eq:general}
|N(\mathcal{A})| \geq \max\{
|N(\mathcal{A}_0)|+\tfrac{r}{n-r}|\mathcal{A}_0|,
|N(\mathcal{A}_1)|+\tfrac{n-r}{r}|\mathcal{A}_1|,
|N(\mathcal{A}_0)|+|N(\mathcal{A}_1)| \}.
\end{equation}
This will be sufficient to prove the inductive step. 
Using (\ref{eq:general}), and applying the induction hypothesis to \(\mathcal{A}_0\) and \(\mathcal{A}_1\), we obtain
\begin{align*}
|N(\mathcal{A})| & \geq \alpha_0 \tbinom{n-1}{r}+ c\sqrt{\tfrac{n-1}{r(n-1-r)}} \alpha_0(1-\alpha_0) \tbinom{n-1}{r} + \tfrac{r}{n-r}\alpha_0 \tbinom{n-1}{r},\\
|N(\mathcal{A})| & \geq \alpha_1 \tbinom{n-1}{r-1}+ c\sqrt{\tfrac{n-1}{(r-1)(n-r)}} \alpha_1(1-\alpha_1) \tbinom{n-1}{r-1} + \tfrac{n-r}{r}\alpha_1 \tbinom{n-1}{r-1}, \\
|N(\mathcal{A})| & \geq \alpha_0 \tbinom{n-1}{r}+ c\sqrt{\tfrac{n-1}{r(n-1-r)}} \alpha_0(1-\alpha_0) \tbinom{n-1}{r}
+ \alpha_1 \tbinom{n-1}{r-1}+ c\sqrt{\tfrac{n-1}{(r-1)(n-r)}} \alpha_1(1-\alpha_1) \tbinom{n-1}{r-1}.
\end{align*}
In terms of the boundary, this becomes
\begin{align*}
|b(\mathcal{A})| & \geq c\sqrt{\tfrac{n-1}{r(n-1-r)}} \alpha_0(1-\alpha_0) \tbinom{n-1}{r} + \tfrac{r}{n-r}\alpha_0 \tbinom{n-1}{r} - \alpha_1\tbinom{n-1}{r-1}\\
&= \left(c\sqrt{\tfrac{n-1}{r(n-1-r)}} (1-\tfrac{r}{n})\alpha_0(1-\alpha_0) + \tfrac{r}{n}(\alpha_0-\alpha_1)\right)\tbinom{n}{r},\\
|b(\mathcal{A})| & \geq c\sqrt{\tfrac{n-1}{(r-1)(n-r)}} \alpha_1(1-\alpha_1) \tbinom{n-1}{r-1} + \tfrac{n-r}{r}\alpha_1 \tbinom{n-1}{r-1} - \alpha_0 \tbinom{n}{r-1}\\
&= \left(c\sqrt{\tfrac{n-1}{(r-1)(n-r)}} \tfrac{r}{n} \alpha_1(1-\alpha_1) + \tfrac{n-r}{n}(\alpha_1 - \alpha_0)\right)\tbinom{n}{r},\\
|b(\mathcal{A})| & \geq c\sqrt{\tfrac{n-1}{r(n-1-r)}} \alpha_0(1-\alpha_0) \tbinom{n-1}{r} + c\sqrt{\tfrac{n-1}{(r-1)(n-r)}} \alpha_1(1-\alpha_1) \tbinom{n-1}{r-1}\\
&= \left(c\sqrt{\tfrac{n-1}{r(n-1-r)}}(1-\tfrac{r}{n}) \alpha_0(1-\alpha_0) + c\sqrt{\tfrac{n-1}{(r-1)(n-r)}} \tfrac{r}{n}\alpha_1(1-\alpha_1)\right)\tbinom{n}{r}.
\end{align*}
To complete the inductive step, we need to deduce that $|b(\mathcal{A})| \ge c\sqrt{\tfrac{n}{r(n-r)}} \alpha(1-\alpha)\tbinom{n}{r}$. Hence, it suffices to prove the following.
\begin{proposition}
\label{prop:technical}
Let \(n \geq 4\), let \(2 \leq r \leq n/2\), let \(0 \leq c \leq 1/5\), and let \(\alpha,\alpha_0,\alpha_1 \in [0,1]\) be such that
\[\alpha = (1-\tfrac{r}{n}) \alpha_0 + \tfrac{r}{n} \alpha_1.\]
Then
\begin{align} \label{eq:max} \max\{
 & c\sqrt{\tfrac{n-1}{r(n-1-r)}} (1-\tfrac{r}{n})\alpha_0(1-\alpha_0) + \tfrac{r}{n}(\alpha_0-\alpha_1), \nonumber \\
 & c\sqrt{\tfrac{n-1}{(r-1)(n-r)}} \tfrac{r}{n} \alpha_1(1-\alpha_1)+ \tfrac{n-r}{n}(\alpha_1 - \alpha_0), \nonumber \\
 & c\sqrt{\tfrac{n-1}{r(n-1-r)}}(1-\tfrac{r}{n}) \alpha_0(1-\alpha_0) + c\sqrt{\tfrac{n-1}{(r-1)(n-r)}} \tfrac{r}{n}\alpha_1(1-\alpha_1)\} \nonumber \\
 & \geq c\sqrt{\tfrac{n}{r(n-r)}} \alpha(1-\alpha).\end{align}
\end{proposition}
\begin{proof}[Proof of Proposition \ref{prop:technical}]
We will show that if \(\alpha_0\) is significantly larger than \(\alpha_1\), then the first term in the above maximum exceeds the right-hand side, if \(\alpha_1\) is significantly larger than \(\alpha_0\), then the second term does so, and otherwise (if \(\alpha_0\) and \(\alpha_1\) are `close' to one another), the third term does so. First, let us introduce some abbreviating notation.

Let
\[\theta_{n}^{(r)} = \sqrt{\tfrac{n}{r(n-r)}}.\]
Let \(\delta = \alpha_0 - \alpha\); then \(\alpha_1 = \alpha - \tfrac{n-r}{r}\delta\). Rewriting (\ref{eq:max}) in terms of the quantities \(\delta\) and \(\theta_i^{(j)}\), we wish to show that
\begin{subequations}
\begin{align} \max\{
 & c (1-\tfrac{r}{n})(\alpha+\delta)(1-\alpha-\delta)\theta_{n-1}^{(r)} + \delta, \label{eq:term1}\\
 & c \tfrac{r}{n} (\alpha-\tfrac{n-r}{r}\delta)(1-\alpha+\tfrac{n-r}{r}\delta)\theta_{n-1}^{(r-1)}- \tfrac{n-r}{r}\delta \label{eq:term2}\\
 & c(1-\tfrac{r}{n}) (\alpha+\delta)(1-\alpha-\delta)\theta_{n-1}^{(r)} + c \tfrac{r}{n}(\alpha-\tfrac{n-r}{r}\delta)(1-\alpha+\tfrac{n-r}{r}\delta)\theta_{n-1}^{(r-1)}\} \label{eq:term3} \\
 & \geq c \alpha(1-\alpha)\theta_{n}^{(r)} \label{eq:term4}.\end{align}\end{subequations}
Before proving the next two claims, we note the estimate
\[ |(\alpha+\eta)(1-\alpha-\eta)-\alpha(1-\alpha)| \leq |\eta|,\]
valid whenever \(\alpha,\alpha+\eta \in [0,1]\). This follows from applying the Mean Value Inequality to the function $f(x)=x(1-x)$ on $[0,1]$:
\[|f(\aA+\eta)-f(\aA)| \le |\eta| \sup_{x \in (0,1)} |f'(x)| = |\eta| \sup_{x \in [0,1]} |1-2x| = |\eta|.\]
We also note that
\[ \tT_{n}^{(r)} \leq \min\{\tT_{n-1}^{(r)},\tT_{n-1}^{(r-1)},1\}.\]
The first two bounds follow from \(\tfrac{n}{n-r} < \tfrac{n-1}{n-1-r}\) and \(\tfrac{n}{r} < \tfrac{n-1}{r-1}\). For the third, fix $n \ge 4$ and note that for \(2 \le r \le n/2\), $\tT_{n}^{(r)}$ is maximised at $r=2$, where it takes the value $\sqrt{\tfrac{n}{2(n-2)}} \le 1$.

We first claim that if \(\delta\) is large and positive, then (\ref{eq:term1}) is at least the right-hand side, (\ref{eq:term4}).
\begin{claim}
\label{claim:largepositive}
If \(\delta \geq 2c  \tfrac{r}{n} \alpha(1-\alpha)\tT_{n}^{(r)}\), then
\[c (1-\tfrac{r}{n})(\alpha+\delta)(1-\alpha-\delta)\tT_{n-1}^{(r)} + \delta \geq c \alpha(1-\alpha)\tT_{n}^{(r)}.\]
\end{claim}
\begin{proof}[Proof of Claim \ref{claim:largepositive}]
Using $\tT_{n}^{(r)} \leq \tT_{n-1}^{(r)}$, it suffices to prove that
\begin{align*} \delta & \geq c \alpha(1-\alpha)\tT_{n}^{(r)} - c (1-\tfrac{r}{n})(\alpha+\delta)(1-\alpha-\delta)\tT_{n}^{(r)}\\
& = c(\tfrac{r}{n} (\alpha+\delta)(1-\alpha-\delta)+ (\alpha(1-\alpha) - (\alpha+\delta)(1-\alpha-\delta)))\tT_{n}^{(r)}.\end{align*}
Since \(|(\alpha+\delta)(1-\alpha-\delta)-\alpha(1-\alpha)| \leq \delta\), this is implied by
\[\delta \geq c(\tfrac{r}{n} (\alpha(1-\alpha)+ \delta) +  \delta)\tT_{n}^{(r)} = c\tfrac{r}{n}\alpha(1-\alpha)\tT_{n}^{(r)} + \delta c (1+\tfrac{r}{n})\tT_{n}^{(r)}.\]
Since \(c \leq 1/5\) and $\tT^{(r)}_n \le 1$, we have 
\[c (1+\tfrac{r}{n})\tT_{n}^{(r)} \leq \tfrac{1}{2},\]
so it suffices that
\[\delta \geq 2c  \tfrac{r}{n} \alpha(1-\alpha)\tT_{n}^{(r)},\]
as required.
\end{proof}

Next, we show that if \(\delta\) is negative, and of large absolute value, then (\ref{eq:term2}) is at least the right-hand side.
\begin{claim}
\label{claim:negative}
If \(\delta \leq -2c \tfrac{r}{n} \alpha(1-\alpha)\tT_{n}^{(r)}\), then
\[c\tfrac{r}{n} (\aA-\tfrac{n-r}{r}\delta)(1-\aA+\tfrac{n-r}{r}\delta)\tT_{n-1}^{(r-1)} - \tfrac{n-r}{r}\delta \ge c \alpha(1-\alpha)\tT_{n}^{(r)}.\]
\end{claim}
\begin{proof}[Proof of Claim \ref{claim:negative}]
Write \(\varepsilon = -\tfrac{n-r}{r}\delta\); we wish to show that
\[c\tfrac{r}{n} (\aA+\varepsilon)(1-\aA-\varepsilon)\tT_{n-1}^{(r-1)} + \varepsilon \ge c \alpha(1-\alpha)\tT_{n}^{(r)}.\]
Using $\tT_{n}^{(r)} \leq \tT_{n-1}^{(r-1)}$, it suffices to prove that
\begin{align*} \varepsilon & \geq c \alpha(1-\alpha)\tT_{n}^{(r)} - c \tfrac{r}{n} (\alpha+\varepsilon)(1-\alpha-\varepsilon)\tT_{n}^{(r)}\\
& = c((1-\tfrac{r}{n})(\alpha+\varepsilon)(1-\alpha-\varepsilon)+ (\alpha(1-\alpha) - (\alpha+\varepsilon)(1-\alpha-\varepsilon)))\tT_{n}^{(r)}.
\end{align*}
Since \(|(\alpha+\varepsilon)(1-\alpha-\varepsilon)-\alpha(1-\alpha)| \leq \varepsilon\), this is implied by
\[\varepsilon \geq c((1-\tfrac{r}{n})( \alpha(1-\alpha)+\varepsilon) + \varepsilon)\tT_{n}^{(r)} = c(1-\tfrac{r}{n})\alpha(1-\alpha)\tT_{n}^{(r)} + \varepsilon c (2-\tfrac{r}{n})\tT_{n}^{(r)}.\]
Since \(c \leq 1/5\) and $\tT_{n}^{(r)} \le 1$, we have 
\[c(2-\tfrac{r}{n})\tT_{n}^{(r)} \leq \tfrac{1}{2}.\]
Thus, it suffices that
\[\varepsilon \geq 2c(1-\tfrac{r}{n})\alpha(1-\alpha)\tT_{n}^{(r)},\]
which is equivalent to our assumption on $\dD$.
\end{proof}

It remains to prove that if \(|\delta| \leq 2c \tfrac{r}{n} \alpha(1-\alpha)\tT_{n}^{(r)}\), then (\ref{eq:term3}) is at least the right-hand side. 

\begin{claim}
\label{claim:smallmodulus}
If \(|\delta| \leq 2c \tfrac{r}{n} \alpha(1-\alpha)\tT_{n}^{(r)}\), then
\[c(1-\tfrac{r}{n}) (\alpha+\delta)(1-\alpha-\delta)\theta_{n-1}^{(r)} + c\tfrac{r}{n}(\alpha-\tfrac{n-r}{r}\delta)(1-\alpha+\tfrac{n-r}{r}\delta)\theta_{n-1}^{(r-1)} \geq c\alpha(1-\alpha)\theta_{n}^{(r)}.\]
\end{claim}
\begin{proof}[Proof of Claim \ref{claim:smallmodulus}]
Multiplying both sides by \(\frac{1}{c\alpha(1-\alpha)}\sqrt{\tfrac{r(n-r)}{n-1}}\), we wish to prove that
\[(1-\tfrac{1}{n-r})^{-1/2} (1-\tfrac{r}{n}) \frac{(\alpha+\delta)(1-\alpha-\delta)}{\alpha(1-\alpha)} + (1-\tfrac{1}{r})^{-1/2} \tfrac{r}{n}\frac{(\alpha-\tfrac{n-r}{r}\delta)(1-\alpha+\tfrac{n-r}{r}\delta)}{\alpha(1-\alpha)} \geq (1+\tfrac{1}{n-1})^{1/2}.\]
Since \((1-x)^{-1/2} \geq 1+x/2\) for all \(x \in [0,1)\), and \((1+x)^{1/2} \leq 1+x/2\) for all \(x \geq 0\), it suffices to prove that
\[W:=(1+\tfrac{1}{2(n-r)})(1-\tfrac{r}{n}) \frac{(\alpha+\delta)(1-\alpha-\delta)}{\alpha(1-\alpha)} + (1+\tfrac{1}{2r}) \tfrac{r}{n}\frac{(\alpha-\tfrac{n-r}{r}\delta)(1-\alpha+\tfrac{n-r}{r}\delta)}{\alpha(1-\alpha)} \geq 1+\tfrac{1}{2(n-1)}.\]

We have
\begin{align*}
W &= (1-\tfrac{r}{n}+\tfrac{1}{2n})\frac{(\alpha+\delta)(1-\alpha-\delta)}{\alpha(1-\alpha)} + (\tfrac{r}{n}+\tfrac{1}{2n})\frac{(\alpha-\tfrac{n-r}{r}\delta)(1-\alpha+\tfrac{n-r}{r}\delta)}{\alpha(1-\alpha)}\\
& = 1 +\tfrac{1}{n}- \tfrac{1-2\alpha}{\alpha(1-\alpha)}\tfrac{n-2r}{2rn}\delta - \left(1-\tfrac{r}{n}+\tfrac{1}{2n}+\tfrac{(n-r)^2}{rn}+\tfrac{(n-r)^2}{2r^2 n}\right)\tfrac{\delta^2}{\alpha(1-\alpha)}\\
& = 1 +\tfrac{1}{n}-T_1-T_2,
\end{align*}
where
\begin{align*}
T_1 &:= \tfrac{1-2\alpha}{\alpha(1-\alpha)}\tfrac{n-2r}{2rn}\delta,\\
T_2 &:= \left(1-\tfrac{r}{n}+\tfrac{1}{2n}+\tfrac{(n-r)^2}{rn}+\tfrac{(n-r)^2}{2r^2 n}\right)\tfrac{\delta^2}{\alpha(1-\alpha)}.
\end{align*}

First, we bound \(|T_1|\) from above. Using the bound on \(|\delta|\), and the fact that \(r \geq 2\), we obtain
\begin{align*}
|T_1| & = \left|\tfrac{1-2\alpha}{\alpha(1-\alpha)}\tfrac{n-2r}{2rn}\delta\right|\\
& \leq \tfrac{n-2r}{2rn}\tfrac{1}{\alpha(1-\alpha)}2c\sqrt{\tfrac{n}{r(n-r)}}\tfrac{r}{n}\alpha(1-\alpha)\\
& = c\tfrac{n-2r}{n^2}\sqrt{\tfrac{n}{r(n-r)}}\\
& < c\tfrac{n-r}{n^2}\sqrt{\tfrac{n}{r(n-r)}}\\
& <\tfrac{c}{\sqrt{r}n}\\
& \leq \tfrac{c}{\sqrt{2}n}.
\end{align*}
We now bound \(|T_2|\) from above. Using the fact that \(\tfrac{(n-r)^2}{nr} \geq \tfrac{1}{2}\) (as \(r \leq n/2\)), and the bound on \(|\delta|\), we obtain
\begin{align*}|T_2| & = \left(1-\tfrac{r}{n}+\tfrac{1}{2n}+\tfrac{(n-r)^2}{rn}+\tfrac{(n-r)^2}{2r^2 n}\right) \tfrac{\delta^2}{\alpha(1-\alpha)}\\
 & \leq \left(1+\tfrac{(n-r)^2}{rn}+\tfrac{(n-r)^2}{4rn}\right)\tfrac{\delta^2}{\alpha(1-\alpha)}\\
& \leq \tfrac{13}{4}\tfrac{(n-r)^2}{rn}\tfrac{\delta^2}{\alpha(1-\alpha)}\\
& \leq \tfrac{13}{4} \tfrac{(n-r)^2}{rn}\tfrac{1}{\alpha(1-\alpha)}\left(2c\sqrt{\tfrac{n}{r(n-r)}}\tfrac{r}{n}\alpha(1-\alpha)\right)^2\\
& = 13 c^2 \tfrac{n-r}{n^2} \alpha(1-\alpha)\\
& < \tfrac{13 c^2}{4n}.
\end{align*}

Putting everything together, since \(n \geq 4\) and \(c \leq 1/5\), we obtain
\[W \geq 1+\tfrac{1}{n} - \tfrac{c}{\sqrt{2}n} -\tfrac{13 c^2}{4n} \geq 1+\tfrac{2}{3n} \geq 1+\tfrac{1}{2(n-1)},\]
as required.
\end{proof}
This completes the proof of Proposition \ref{prop:technical}, proving Theorem \ref{thm:main}.
\end{proof}\end{proof}
\section{Conclusion}
We believe that a proof of Conjecture \ref{conj:main} would require new techniques. Interestingly, as Bollob\'as and Leader point out in \cite{leader}, in the case when \(n\) is even and \(r = n/2\), the only sets \(\mathcal{B}_{k,l}\) one needs to consider are the `balls' \(\mathcal{B}_{n/2,l}\), which {\em are} nested. Hence, a proof using compressions is not ruled out. Since the conjectured extremal sets are not nested for \(r < n/2\), such a proof would almost certainly `remain within the middle layer', which ours does not.

We remark that the edge-isoperimetric problem for \(\Gamma^{(r)}_n\) remains open. A conjecture of Kleitman \cite{kleitman} on the exact solution (for all set-sizes) was disproved by Ahlswede and Cai \cite{ac}; at present, to the best of our knowledge, there is no general conjecture as to the exact solution. An approximate result has been obtained by Harper \cite{harper3}.


\begin{thebibliography}{99}
\bibitem{ac} R. Ahlswede, N. Cai, A counterexample to Kleitman's conjecture concerning an edge-isoperimetric problem, {\em Combin. Probab. and Comput.} 8 (1999), pp. 301--305.
\bibitem{bernstein} A. J. Bernstein, Maximally connected arrays on the \textit{n}-cube, {\em SIAM Journal on Applied Mathematics} 15 (1967), pp. 1485-1489.
\bibitem{bezrukov} S. Bezrukov, Isoperimetric problems in discrete spaces, in: {\em Extremal Problems for Finite Sets.}, Bolyai Soc. Math. Stud. 3, Budapest 1994, pp. 59--91.
\bibitem{bobkov} S. G. Bobkov, An isoperimetric inequality on the discrete cube, and an elementary proof of the isoperimetric inequality in Gauss space, {\em The Annals of Probability} 25 (1997), No. 1, pp. 206-214.
\bibitem{bollobas} B. Bollob\'as, {\em Combinatorics: Set systems, hypergraphs, families of vectors and combinatorial probability}, CUP 1986.
\bibitem{leader} B. Bollob\'as, I. Leader, Isoperimetric inequalities for \(r\)-sets, {\em Combinatorics, Probability and Computing} 13 (2004), pp. 277-279.
\bibitem{harper2} L. H. Harper, Optimal numberings and isoperimetric problems on graphs, 
{\em J. Combin. Theory} 1 (1966), pp.385-393.
\bibitem{harper} L. H. Harper, Optimal assignments of numbers to vertices, {\em SIAM Journal on Applied Mathematics} 12 (1964), pp. 131-135.
\bibitem{harper3} L. H. Harper, On a problem of Kleitman and West, {\em Discrete Math.} 93 (1991), pp. 169--182.
\bibitem{hart} S. Hart, A note on the edges of the \(n\)-cube, \emph{Discrete Mathematics} 14 (1976), pp. 157--163.
\bibitem{kleitman} D. J. Kleitman, Extremal hypergraph problems, in {\em Surveys in Combinatorics 1979}, London Math. Soc. Lecture Note Series 38, Cambridge University Press, pp. 44--65.
\bibitem{compressions} I. Leader, Discrete Isoperimetric Inequalities, in {\em Probabilistic Combinatorics and its Applications}, ed. B. Bollob\'as and F.K.R. Chung, American Mathematical Society 1991.
\bibitem{lindsey} J. H. Lindsey, II, Assignment of numbers to vertices, {\em American Mathematical Monthly} 71 (1964) 508-516. 
\bibitem{talagrand} M. Talagrand, Concentration of measure and isoperimetric inequalities in product spaces, 
{\em Publ. Math. I.H.E.S.} 81, 1995, pp. 73--203.
\bibitem{talagrand2} M. Talagrand, Isoperimetry, Logarithmic Sobolev inequalities on the discrete cube, and Margulis' graph connectivity theorem, {\em GAFA} 3 (1993), No. 3, pp. 295-314.
\end{thebibliography}
\end{document}